\documentclass[11pt,bezier]{article}
\pdfoutput=1
\usepackage{amsmath}
\usepackage{amsfonts,amsthm,amssymb}
\usepackage{exscale}
\usepackage{relsize}
\usepackage{graphicx}
\usepackage{epstopdf}
\usepackage{caption}
\usepackage{float}

\usepackage{amssymb}
\usepackage{amsmath}
\newtheorem{lemma}{Lemma}[section]
\newtheorem{theorem}{Theorem}[section]
\newtheorem{corollary}{Corollary}[section]

\usepackage{graphicx}
\makeatletter
\newcommand{\Rmnum}[1]{\expandafter\@slowromancap\romannumeral #1@}
\makeatother
\theoremstyle{definition}

\usepackage[justification=centering]{caption}
\allowdisplaybreaks[3]
\begin{document}
\title{The 3-restricted Edge-Connectivity of Strong Product Graphs\footnote{The research is supported by Natural Science Foundation of Xinjiang Uygur Autonomous Region (2025D01E02) and National Natural Science Foundation of China (12261086).}}
\author{Wenxin Wang$^{a}$, Yingzhi Tian$^{a}$\footnote{Corresponding author. E-mail: 107552300678@stu.xju.edu.cn (W. Wang); tianyzhxj@163.com (Y. Tian); Piaoxiangxing@163.com (J. Wang).}, Jing Wang$^{b}$\\
{\small $^{a}$College of Mathematics and System Sciences, Xinjiang
University, Urumqi, Xinjiang 830046, PR China}\\
{\small $^{b}$Xinjiang Railway Vocational and Technical College, Urumqi, Xingjiang 830011, PR China}}

\date{}

\maketitle

\begin{sloppypar}

\noindent{\bf Abstract }An edge subset \( S \subseteq E(G) \) is called a 3-restricted edge-cut if $G-S$ is disconnected and each component of \( G - S \) contains at least three vertices. The 3-restricted edge-connectivity of a graph \( G \), denoted by \( \lambda_3(G) \), is defined as the minimum cardinality among all 3-restricted edge-cuts if there are at least one; otherwise, \( \lambda_3(G) = +\infty \).
It is proved that $\lambda_3(G)\leq\xi_3(G)$ if $G$ has a 3-restricted edge-cut, where $\xi_3(G) = \min \{ |[X, V(G) \setminus X]_G||X \subseteq V(G),|X| = 3 \text{ and } G[X] \text{ is connected}\}.$
If \( \lambda_3(G) = \xi_3(G) \), then \( G \) is said to be maximally 3-restricted edge-connected. The strong product of graphs \( G \) and \( H \), denoted by \( G \boxtimes H \), is the graph with the vertex set $ V(G)\times V(H) $ and the edge set $ \{(x_{1},y_{1})(x_{2},y_{2})|x_{1}=x_{2}\text{ and }y_{1}y_{2}\in E(H);\text{ or }y_{1}=y_{2} $ and $ x_{1}x_{2}\in E(G) $; or $ x_{1}x_{2}\in E(G) $ and $ y_{1}y_{2}\in E(H)\}$.
In this paper, we prove that \( G \boxtimes C_{n} \) is maximally 3-restricted edge-connected, and determine the 3-restricted edge-connectivity of \( G \boxtimes K_{n} \), where \( G \) is a maximally edge-connected graph, \( C_{n} \) and \( K_{n} \) are the cycle and the complete graph of order \( n \), respectively.


\section{Introduction}
The topology structure of an interconnection network is usually modeled by a graph \( G = (V(G), E(G)) \), where \( V(G) \) represents the set of nodes in the network, and \( E(G) \) represents the set of links between these nodes.  Reliability and fault-tolerance are two crucial factors in designing a network. The edge-connectivity of the graph $G$ is used to measure the reliability and fault-tolerance of the corresponding network. For the graph $G = (V(G), E(G))$, its order and size are denoted by \( n(G) = |V(G)| \) and \( e(G) = |E(G)|\), respectively. For a vertex $u \in V(G)$, its $neighborhood$ $N_G(u)$ is $\{ v \in V(G) \mid v \text{ is adjacent to } u \}$, and its $degree$ $d_G(u)$ is the number of edges incident with $u$. The $minimum$ $degree$ $\delta(G)$ and the $maximum$ $degree$ $\Delta(G)$ of $G$ are $\min_{u\in V(G)}d_{G}(u)$ and $\max_{u\in V(G)}d_{G}(u)$, respectively. The $induced$ $subgraph$ $G[U]$ of a vertex subset $U \subseteq V(G)$, is the graph with the vertex set $U$, where two vertices in $U$ are adjacent in $G[U]$ if and only if they are adjacent in $G$. For graph-theoretical terminologies and notation not defined here, we follow \cite{Bondy}.

A graph \( G \) is $connected$ if there exists a path between every pair of distinct vertices; otherwise, $G$ is $disconnected$. Each maximal connected subgraph of $G$ is called its $component$. An edge subset \( S \subseteq E(G) \) is called an $edge$-$cut$ if \( G - S \) is disconnected. The $edge$-$connectivity$ \(\lambda(G)\) of $G$ is defined as the minimum cardinality among all edge-cuts. It is well-known that \(\lambda(G) \leq \delta(G)\). So the graph \( G \) is called $maximally$ $edge$-$connected$ if \( \lambda(G) = \delta(G) \); and \( G \) is called $super$ $edge$-$connected$ if every minimum edge-cut isolates one vertex.

One deficiency of the edge-connectivity is that it only considers when the remaining graph is not connected, but not the properties of the components. To address this limitation, Esfahanian and Hakimi \cite{Esfahanian} proposed the concept of restricted edge-connectivity, which is a type of conditional edge-connectivity initially introduced by Harary \cite{Harary}. An edge set \( S \subseteq E(G) \) is termed a $restricted$ $edge$-$cut$ if \( G - S \) is disconnected and each component of \( G - S \) contains at least two vertices. The $restricted$ $edge$-$connectivity$ of a graph \( G \), denoted by \( \lambda_2(G) \), is defined as the minimum cardinality among all restricted edge-cuts if there are at least one; otherwise, \( \lambda_2(G) = +\infty \). For an edge $e = uv \in E(G)$, its $edge$-$degree$ $\xi_G(e)$ is $d_G(u) + d_G(v) - 2$, and the $minimum$ $edge$-$degree$ of $G$ is $\xi(G) = \min \{ \xi_G(e) \mid e \in E(G) \}$. In graph $G$, two edges are defined as $non$-$adjacent$ $edges$ if they have no common end-vertices. Esfahanian and Hakimi \cite{Esfahanian} demonstrated that if \( G \) is not a star graph \( K_{1,n-1} \) and has at least four vertices, then \( \lambda_2(G) \leq \xi(G) \). So the graph \( G \) is called $maximally$ $restricted$ $edge$-$connected$ if \( \lambda_2(G) = \xi(G) \); and the graph $G$ is called $super$ $restricted$ $edge$-$connected$ if every minimum restricted edge-cut isolates an edge.

Generally, the concept of $k$-restricted edge-connectivity can be defined as follows. If an edge subset \( S \subseteq E(G) \) satisfies that \( G - S \) is disconnected and each component of \( G - S \) has at least $k$ vertices, then $S$ is called a $k$-$restricted$ $edge$-$cut$. The $k$-$restricted$ $edge$-$connectivity$ \(\lambda_k(G)\) is defined as the minimum cardinality among all $k$-restricted edge-cuts in $G$ if there exists one; otherwise, \( \lambda_k(G) = +\infty \).

For two nonempty subsets \( X, Y \subseteq V(G) \), \( [X, Y]_G \) represents the set of edges with one endpoint in \( X \) and the other in \( Y \). When \( Y = V(G) \setminus X \), the edge set \( [X, Y]_G \) is denoted as \( \partial_G(X) \). Let \( \xi_3(G) = \min \{ |\partial_G(X)||X \subseteq V(G), |X| = 3 \text{ and } G[X] \text{ is connected} \} \). Wang and Li \cite{Wang1} provided necessary and sufficient conditions for the existence of 3-restricted edge-cuts in graphs and demonstrated that \( \lambda_3(G) \leq \xi_3(G) \) holds when there is at least one 3-restricted edge-cut in $G$. So the graph \( G \) is termed $maximally$ $3$-$restricted$ $edge$-$connected$ if \( \lambda_3(G) = \xi_3(G) \); and the graph \( G \) is called $super$ $3$-$restricted$ $edge$-$connected$ if every minimum 3-restricted edge-cut isolates a component of order three.

The concept of graph products serves as a potent tool for constructing larger graphs from smaller components. There are several common graph products, including the Cartesian product, the direct product, and the strong product, etc. Given two graphs \( G \) and \( H \), the vertex sets of the $Cartesian$ $product$ \( G \square H \), the $direct$ $product$ \( G \times H \) and the $strong$ $product$ \( G \boxtimes H \) are all defined as \( V(G) \times V(H) \). For two distinct vertices \( (x_1, y_1) \) and \( (x_2, y_2) \),
they are adjacent in \( G \square H \) if and only if \( x_1 = x_2 \) and \( y_1y_2 \in E(H) \), or \( y_1 = y_2 \) and \( x_1x_2 \in E(G) \);
they are adjacent in \( G \times H \) if and only if \( x_1x_2 \in E(G) \) and \( y_1y_2 \in E(H) \);
they are adjacent in \( G \boxtimes H \) if and only if \( x_1 = x_2 \) and \( y_1y_2 \in E(H) \), or \( y_1 = y_2 \) and \( x_1x_2 \in E(G) \), or both \( x_1x_2 \in E(G) \) and \( y_1y_2 \in E(H) \). It is evident that the edge set of \( G \boxtimes H \) comprises the union of the edge sets of \( G \square H \) and \( G \times H \).

Klav{\v{z}}ar and {\v{S}}pacapan \cite{Klavzar} determined the edge-connectivity of the Cartesian product of any two nontrivial connected graphs. Wang, Ou and Zhu \cite{Wang2} studied the restricted edge-connectivity of the Cartesian product of two regular graphs. Ou \cite{Ou1} obtained the restricted edge connectivity of the Cartesian product of two maximally edge-connected regular graphs.

Bre{\v{s}}ar and {\v{S}}pacapan \cite{Bresar2} provided some lower and upper bounds on the edge-connectivity of the direct product of two nontrivial connected graphs. Cao, Brglez, {\v{S}}pacapan and Vumar \cite{Cao} obtained the edge-connectivity of the direct product of a nontrivial graph with a complete graph. {\v{S}}pacapan \cite{Spacapan} not only determined the edge-connectivity of the direct product of two graphs but also characterized the structure of the minimum edge-cut for the direct product of two graphs.
Ma, Wang and Zhang \cite{Ma} determined the restricted edge-connectivity of the direct product of a nontrivial connected graph with a complete graph. Bai, Tian and Yin [1] further gave sufficient conditions for the direct product of a nontrivial graph with a complete graph to be super restricted edge-connected. Guo, Hu, Yang and Zhao \cite{Guo} investigated the restricted edge-connectivity of the direct product of a connected graph \( G \) that satisfies \( |V(G)| \leq n \) or \( \Delta(G) \leq n - 1 \) with an odd cycle of order $n$.

The edge-connectivity of the strong product of two nontrivial connected graphs was established by Bre\v{s}ar and \v{S}pacapan \cite{Bresar1}. Subsequently, Ou and Zhao \cite{Ou2} focused on the restricted edge-connectivity of the strong product of two triangle-free connected graphs. Wang, Mao, Ye and Zhao \cite{Wang3} derived an expression for the restricted edge-connectivity of the strong product of two maximally restricted edge-connected graphs.

Let $C_n$ and $K_n$ denote the $cycle$ and the $complete$ $graph$ of order $n$, respectively. Ye and Tian \cite{Ye} gave the restricted edge-connectivity of the strong product of a nontrivial connected graph with a cycle and a complete graph.

\begin{theorem}(\cite{Ye}\label{1})
Let $G$ be a connected nontrivial graph of order $m$. Then $\lambda_2(G \boxtimes C_n) = \min\{3n\lambda(G), 2(m + 2e(G)), 6\delta(G) + 2\}$, where $n \geq 3$.
\end{theorem}

\begin{theorem}(\cite{Ye}\label{2})
Let $G$ be a connected nontrivial graph of order $m$. Then $\lambda_2(G \boxtimes K_n) = \min\{n^2\lambda(G), (n - 1)(m + 2e(G)), 2n\delta(G) + 2n - 4\}$, where $n \geq 4$.
\end{theorem}

Motivated by the results above, in this paper, we investigate the 3-restricted edge-connectivity of the strong product of a maximally edge-connected graph with a cycle and a complete graph. As corollaries, we establish sufficient conditions for these strong product graphs to be maximally 3-restricted edge-connected. The subsequent section will introduce relevant definitions and lemmas. The main results will be presented in Section 3. The concluding remarks are given in the last section.

\section{Preliminary}
 Let \( G \) and \( H \) be two graphs with $m$ and $n$ vertices, respectively. Set $V(G) = \{x_1, x_2, \ldots, x_m\}$ and $V(H) = \{y_1, y_2, \ldots, y_n\}$. Denote \( \mathcal{G} = G \boxtimes H \). By the definition of the strong product, for any $(x, y) \in V(\mathcal{G})$, we have $d_{\mathcal{G}}((x, y))=d_G(x)+d_{H}(y)+d_G(x)d_{H}(y)$. Define the projection \( p \) on \( V(\mathcal{G})  \) as follows: for any \( (x, y) \in V(\mathcal{G}) \), \( p(x, y) = y \). For any vertex \( x \in V(G) \), the \( H \)-layer of \( x \) in \( \mathcal{G} \) is defined as $H^x = \{ (x, y) \in V(\mathcal{G}) \mid y \in V(H) \}.$
Analogously, for any vertex \( y \in V(H) \), the \( G \)-layer of \( y \) in \( \mathcal{G} \) is given by $G^y = \{ (x, y) \in V(\mathcal{G}) \mid x \in V(G) \}.$ Clearly, \( H^{x} \cong H \) and \( G^{y} \cong G \).

Let \( H \) be a connected graph. Define \( K_2 \odot H = K_2 \boxtimes H - E(\{a\} \boxtimes H) - E(\{b\} \boxtimes H) \), where \( V(K_2) = \{a, b\} \). Here, \( \{a\} \boxtimes H \) denotes the strong product of \( H \) with the complete graph containing only one vertex \( a \), and \( \{b\} \boxtimes H \) denotes the strong product of \( H \) with the complete graph containing only one vertex \( b \). By this definition, \( K_2 \odot H \) is connected if and only if \( H \) is connected.

\begin{lemma}(\cite{Bondy}\label{4})
	For any graph $G$,
$$\sum_{u \in V(G)} d(u) = 2e(G).$$
\end{lemma}

\begin{lemma}(\cite{Yang}\label{4})
	Let $H$ be a connected graph and $S$ be an edge-cut of $K_{2} \odot H$, where $K_{2}=\{a,b\}$. If the vertices of $\{a\} \boxtimes H$ are in different components of $K_{2} \odot H - S$, and the vertices of $\{b\} \boxtimes H$ are also in different components of $K_{2} \odot H - S$, then $|S| \geqslant 2\lambda(H)$.
\end{lemma}

\begin{lemma}(\cite{Wang1}\label{5})
	The graph $G$ has a 3-restricted edge-cut if and only if $G$ has two vertex-disjoint paths of order 3.
\end{lemma}

\begin{lemma}(\cite{Wang1}\label{6})
	Let \( G \) be a connected graph of order at least 6. If $G$ has a 3-restricted edge-cut, then $\lambda_{3}(G)\leq{\xi_3(G)}$.
\end{lemma}

\begin{lemma}
	Let \( G \) and \( H \) be connected nontrivial graphs of order \( m\geq3 \) and \( n \geq 3 \), respectively. Then
	\[
	\lambda_3(G \boxtimes H) \leq \min\{(n + 2e(H))\lambda(G), (m + 2e(G))\lambda(H)\}.
	\]
\end{lemma}

\noindent{\bf Proof.}
	Let $[X,\overline{X}]_G$ be a minimum edge-cut of $G$. Then $[X \times V(H), \overline{X} \times V(H)]_{G \boxtimes H}$ is a 3-restricted edge-cut of $G \boxtimes H$, and its size is $(n + 2e(H))\lambda(G)$. Hence, $\lambda_3(G \boxtimes H) \leq (n + 2e(H))\lambda(G)$. Analogously, let $[Y,\overline{Y}]_H$ be a minimum edge-cut of $H$. Then $[V(G) \times Y, V(G) \times \overline{Y}]_{G \boxtimes H}$ is also a 3-restricted edge-cut of $G \boxtimes H$, and its size is $(m + 2e(G))\lambda(H)$. Thus, $\lambda_3(G \boxtimes H) \leq (m + 2e(G))\lambda(H)$. Therefore, $\lambda_3(G \boxtimes H) \leq \min\{(n + 2e(H))\lambda(G), (m + 2e(G))\lambda(H)\}$.
$\square$

\begin{lemma}
Let \( G \) and \( H \) be connected nontrivial graphs of order \( m \geq 3 \) and \( n \geq 3 \), respectively. Let \( S \) be a minimum 3-restricted edge-cut of \( G \boxtimes H \), and \( D_1 \) and \( D_2 \) be the two components of \( G \boxtimes H - S \). If each vertex \( x \in V(G) \) satisfies \( H^x \cap D_1 \neq \emptyset  \) and \( H^x \cap D_2 \neq \emptyset  \), or each vertex \( y \in V(H) \) satisfies \( G^y \cap D_1 \neq \emptyset  \) and \( G^y \cap D_2 \neq \emptyset  \), then \( |S| \geq (m + 2e(G))\lambda(H) \) or \( |S| \geq (n + 2e(H))\lambda(G) \).
\end{lemma}

\noindent{\bf Proof.} Assume that for each vertex \( x_i \in V(G) \), we have \( H^{x_i} \cap D_1 \neq \emptyset \) and \( H^{x_i} \cap D_2 \neq \emptyset  \). For \( 1 \leq i \leq m \), let \( Y_i = V(H^{x_i}) \cap V(D_1) \) and \( \overline{Y}_i = V(H^{x_i}) \setminus Y_i \). By Lemma 2.2, for any edge \( uv \in E(G) \), we have $|E(G[\{u, v\}] \odot H) \cap S| \geq 2\lambda(H).$
Therefore, we obtain

\begin{align*}
|S| &\geq \sum_{i=1}^m | [Y_i,\overline{Y}_i]_{H^{x_i}} | + \sum_{e=uv\in E(G)} | E(G[\{u,v\}] \odot H) \cap S | \\
&\geq m \cdot \lambda(H) + e(G) \cdot 2\lambda(H) \\
&= (m + 2e(G))\lambda(H).
\end{align*}

Analogously, if each vertex \( y_j \in V(H) \) satisfies \( G^{y_j} \cap D_1 \neq \emptyset  \) and \( G^{y_j} \cap D_2 \neq \emptyset  \), then we obtain $|S| \geq (n + 2e(H))\lambda(G).$
$\square$

\section{Main Results}
\begin{theorem}
	Let $G$ be a maximally edge-connected graph of order $m$ $(\geq5)$. If $\delta(G)\geq2$ and $n\geq4$, then
\[
\lambda_3(G \boxtimes C_n)=\xi_3(G \boxtimes C_n)=\begin{cases}9\delta(G), & \text{if } \xi(G)=2\delta(G)-2;\\9\delta(G)+2, & \text{otherwise,}\end{cases}
\]
that is, $G \boxtimes C_n$ is maximally 3-restricted edge-connected.
\end{theorem}

\noindent{\bf Proof.} Let $\mathcal{G} = G \boxtimes C_n$. It is routine to verify that $\xi_3(\mathcal{G})=9\delta(G)$ if $\xi(G)=2\delta(G)-2$; $\xi_3(\mathcal{G})=9\delta(G)+2$  if $\xi(G)>2\delta(G)-2$. Thus, by Lemma 2.4, we have $\lambda_3(\mathcal{G}) \leq9\delta(G)$ if $\xi(G)=2\delta(G)-2$; $\lambda_3(\mathcal{G}) \leq9\delta(G)+2$ otherwise.

Now, it suffices to prove that \[
	\lambda_3(G \boxtimes C_n) \geq \begin{cases}9\delta(G) ,\text{ if }\xi(G)=2\delta(G)-2;\\9\delta(G)+2,\text{ otherwise},&\end{cases}
	\]

Let $S$ be a minimum 3-restricted edge-cut of $\mathcal{G}$. Then $\mathcal{G} - S$ has two components $D_1$ and $D_2$, where $|V(D_1)| \geq 3$ and $|V(D_2)| \geq 3$.

	If each vertex $x_i \in V(G)$ satisfies $C_n^{x_i} \cap D_1 \neq \emptyset $ and $C_n^{x_i} \cap D_2 \neq \emptyset $, or each vertex $y_j \in V(C_n)$ satisfies $G^{y_j } \cap D_1 \neq \emptyset $ and $G^{y_j } \cap D_2 \neq \emptyset $,
then by Lemma 2.6, we have
$|S|\geq(m+2e(G))\lambda(C_n)$ or $|S|\geq(n+2e(C_n))\lambda(G)$. By Lemma 2.1, we have $2e(G)=\sum_{x \in V(G)} d(x) \geq  m\delta(G)$. Since $\lambda(G)=\delta(G), \lambda(C_n) =2$, $e(C_n) =n$, $m\geq5$ and $n\geq4$,
we have $|S|\geq\min \{(m+2e(G))\lambda(C_n), (n+2e(C_n))\lambda(G)\}=\min\{2(m+2e(G)),\ 3n\delta(G) \}\geq\min\{2(m+m\delta(G)),\ 3n\delta(G) \}\geq\min\{10(\delta(G)+1),\ 12\delta(G) \}>9\delta(G)+2.$

	Now we consider that there exists a vertex $x_a \in V(G)$ and a vertex $y_b \in V(C_n)$ such that $C_n^{x_a} \cap D_1 = \emptyset $ and $G^{y_b} \cap D_1 = \emptyset $, or $C_n^{x_a} \cap D_2 = \emptyset $ and $G^{y_b} \cap D_2 = \emptyset $. Without loss of generality, assume that there exists a vertex $x_a \in V(G)$ and a vertex $y_b \in V(C_n)$ such that $C_n^{x_a} \cap D_1 = \emptyset $ and $G^{y_b} \cap D_1 = \emptyset $. By this assumption, $ V(C_n^{x_a}) $ and $ V(G^{y_b}) $ are contained in $ D_2 $. Let $ p(V(D_1)) = \{y_{s+1}, y_{s+2}, \ldots, y_{s+t}\} $, where the addition is modulo $n$. Without loss of generality, assume $ s + t < b $. For $ 1 \leq i \leq t $, let $ X_i = V(G^{y_{s+i}}) \cap V(D_1) $ and $ \overline{X}_i = V(G^{y_{s+i}}) \setminus X_i $. Since $G^{y_{s+i}} \cong G$ and $\lambda(G) = \delta(G)$, it follows that $
|[X_i, \overline{X}_i]_{G^{y_{s+i}}}| \geq \lambda(G^{y_{s+i}}) = \lambda(G) = \delta(G)$.  By the definition of the strong product, for any \((x,y_{s + 1}) \in X_1\), we have \( |[ \{ (x,y_{s + 1}) \}, V(G^{y_s}) ]_\mathcal{G}|  = d_G(x) + 1 \geq \delta(G) + 1\). Thus \(|[X_1, V(G^{y_s})]_\mathcal{G}| \geq |X_1|(\delta(G) + 1)\). Analogously, $|[X_t,V(G^{y_{s + t+1}})]_\mathcal{G}|\geq\vert X_t\vert(\delta(G)+1)$. Therefore,	
\begin{align}
|S| &\geq \sum_{i=1}^{t}|[X_{i},\overline{X}_{i}]_{G^{y_{s+i}}}|
          + |[X_{1}, G^{y_{s}}]_{\mathcal{G}}|
          + |[X_{t}, G^{y_{s+t+1}}]_{\mathcal{G}}|
          + \sum_{i=1}^{t-1}|[X_{i},\overline{X}_{i+1}]_{\mathcal{G}}|
          + \sum_{i=1}^{t-1}|[X_{i+1},\overline{X}_{i}]_{\mathcal{G}}|\nonumber \\
    &\geq t\delta(G)
          + (|X_{1}| + |X_{t}|)(\delta(G) + 1)
          + \sum_{i=1}^{t-1}|[X_{i},\overline{X}_{i+1}]_{\mathcal{G}}|
          + \sum_{i=1}^{t-1}|[X_{i+1},\overline{X}_{i}]_{\mathcal{G}}|.
\end{align}

	By Lemma 2.2, we have $\sum_{i=1}^{t-1}|[X_{i},\overline{X}_{i+1}]_{\mathcal{G}}|+\sum_{i=1}^{t-1}|[X_{i+1},\overline{X}_{i}]_{\mathcal{G}}|
\geq2(t-1)\lambda(G) =2(t-1)\delta(G).$
	Thus, by (1), we have
\begin{equation}	
\begin{aligned}
	|S|\geq t \delta(G)+(|X_{1}|+|X_{t}|)(\delta(G)+1)+2(t-1)\delta(G).
	\end{aligned}
\end{equation}

	If $|V(D_1)|=3$ or $|V(D_2)|=3$, then we have $|S|\geq\xi_3(\mathcal{G})$. So we assume $|V(D_1)| \geq 4$ and $|V(D_2)| \geq 4$ in the following proof.

	\noindent{\bf Case 1.} $t=1$.

	When $t = 1$, it follows that $V(D_1)\subseteq V(G^{y_{s + 1}})$, which implies $V(D_1)=X_1$, Since $D_1$ is a connected graph and $|V(D_1)| \geq 4$, we have $G^{y_{s + 1}}[X_1]$ is connected and $|X_1|\geq4$. By (2), we have
	\[
	\begin{aligned}
		|S|&=\delta(G)+8(\delta(G)+1)\\
		&=9\delta(G)+8\\
		&>9\delta(G)+2.
	\end{aligned}
	\]

	\noindent{\bf Case 2.} $t=2$.

	When $t = 2$, we have $V(D_1)=X_{1}\cup X_{2}$. Since $|V(D_1)| \geq 4$, it follows that $|X_{1}|+|X_{2}|\geq4$.

	If $|X_{1}|+|X_{2}|=4$, then
\[
\begin{alignedat}{2}
|[X_{1}, \overline{X}_{2}]_{\mathcal{G}}|+|[X_{2}, \overline{X}_{1}]_{\mathcal{G}}|
    &\geq|X_{1}|(\delta(G)+1-|X_{2}|)+|X_{2}|(\delta(G)+1-|X_{1}|) \\
    &=(\delta(G)+1)(|X_{1}|+|X_{2}|)-2|X_{1}|X_{2}| \\
    &\geq(\delta(G)+1)(|X_{1}|+|X_{2}|)-2(\frac{|X_{1}|+|X_{2}|}{2})^{2} \\
    &=4\delta(G)-4.
\end{alignedat}
\]
	By (1), we have
$|S|\geq2\delta(G) + 4(\delta(G) + 1) + 4\delta(G)-4
	=10\delta(G)\geq9\delta(G)+2.$

	If $|X_{1}|+|X_{2}|\geq5$, by (2), we have $|S|\geq2\delta(G) + 5(\delta(G) + 1) + 2\delta(G)=9\delta(G)+5>9\delta(G)+2.$

	\noindent{\bf Case 3.} $3\leq t\leq n - 1$.

	Since $|X_{1}|+|X_{t}|\geq2$, by (2), we have $|S|\geq3\delta(G) + 2(\delta(G) + 1) + 4\delta(G)
	=9\delta(G)+2.$
$\square$

\begin{theorem}
	Let \( G \) be a maximally edge-connected graph of order $m$ $(\geq 3)$. Then
	\[
	\lambda_3(G \boxtimes K_n) = \min \{
	n^{2}\delta(G),\
	(n-1)(m + 2e(G)),\
	3n\delta(G) + 3n - 9
	\},
	\]
	where $n \geq 4$.
\end{theorem}

\noindent{\bf Proof.} Let $\mathcal{G} = G \boxtimes K_n$.
Since $\lambda(G)=\delta(G), \lambda(K_n) = n - 1$ and $e(K_n) = \frac{n(n - 1)}{2}$, by Lemma 2.5, we have $\lambda_3(\mathcal{G}) \leq \min\{ (n + 2e(K_n))\lambda(G),\ (m + 2e(G))\lambda(K_n) \}=\min \{ n^2\delta(G),\ (n - 1)(m + 2e(G)) \}.$ By Lemma 2.4, $\lambda_3(\mathcal{G}) \leq \xi_3(\mathcal{G})=3n\delta(G) + 3n - 9.$
Therefore, $\lambda_3(\mathcal{G}) \leq \min \{
	n^{2}\delta(G),\
	(n-1)(m + 2e(G)),\
	3n\delta(G) + 3n - 9
	\}.$

Now, it suffices to prove that $\lambda_3(\mathcal{G}) \geq \min \{
n^{2}\delta(G),\
(n-1)(m + 2e(G)),\
3n\delta(G) + 3n - 9
\}.$
Let $S$ be a minimum 3-restricted edge-cut of $\mathcal{G}$. Then $\mathcal{G} - S$ has two components $D_1$ and $D_2$, where $|V(D_1)| \geq 3$ and $|V(D_2)| \geq 3$.

	If each vertex $x_i \in V(G)$ satisfies $K_n^{x_i} \cap D_1 \neq \emptyset $ and $K_n^{x_i} \cap D_2 \neq \emptyset $, or each vertex $y_j \in V(K_n)$ satisfies $G^{y_j} \cap D_1 \neq \emptyset $ and $G^{y_j} \cap D_2 \neq \emptyset $,
then by Lemma 2.6, we have $|S| \geq (m + 2e(G))\lambda(K_n)=(n - 1)(m + 2e(G))$ or $|S| \geq (n + 2e(K_n))\lambda(G)=n^{2}\delta(G); $ that is, $|S| \geq \min \{(n-1)(m + 2e(G)), \ n^{2}\delta(G)\}.$
	
Now we consider that there exists a vertex $x_a \in V(G)$ and a vertex $y_b \in V(K_n)$ such that $K_n^{x_a} \cap D_1 = \emptyset $ and $G^{y_b} \cap D_1 = \emptyset $, or $K_n^{x_a} \cap D_2 = \emptyset $ and $G^{y_b} \cap D_2 = \emptyset $. Without loss of generality, assume that there exists a vertex $x_a \in V(G)$ and a vertex $y_b \in V(K_n)$ such that $K_n^{x_a} \cap D_1 = \emptyset $ and $G^{y_b} \cap D_1 = \emptyset $. By this assumption, $ V(K_n^{x_a}) $ and $ V(G^{y_b}) $ are contained in $ D_2 $. Since any two distinct vertices in $ K_n $ are adjacent, by renaming the vertices of $ V(K_n) $, let $ p(V(D_1)) = \{y_{s+1}, y_{s+2}, \ldots, y_{s+t}\} $. Without loss of generality, assume $ s + t < b $. For $ 1 \leq i \leq t $, let $ X_i = V(G^{y_{s+i}}) \cap V(D_1) $ and $ \overline{X}_i = V(G^{y_{s+i}}) \setminus X_i $. Since $G^{y_{s+i}} \cong G$ and $\lambda(G) = \delta(G)$, it follows that $|[X_i, \overline{X}_i]_{G^{y_{s+i}}}| \geq \lambda(G^{y_{s+i}}) = \lambda(G) = \delta(G).$ Let $ Y = V(K_n) \setminus p(V(D_1)) $. By the definition of the strong product, for any $ y \in Y $, we have $ |[X_i, G^y]_{\mathcal{G}}| \geq |X_i|(\delta(G) + 1) $. Therefore,

\[
	\begin{aligned}
		|S| &\geq \sum_{i=1}^{t}|[X_{i}, \overline{X}_{i}]_{G^{y_{s+i}}}| + \sum_{i=1}^{t} \sum_{y \in Y}|[X_{i}, G^{y}]_{\mathcal{G}}| + \sum_{i=1}^{t} \sum_{\substack{j=1 \\ j \neq i}}^{t}|[X_{i}, \overline{X}_{j}]_{\mathcal{G}}|\\
\end{aligned}
	\]
\begin{equation}
\begin{aligned}
	&\geq t \delta(G) + \sum_{i=1}^{t}|X_{i}|(\delta(G)+1)(n-t) + \sum_{i=1}^{t} \sum_{\substack{j=1 \\ j \neq i}}^{t}|[X_{i}, \overline{X}_{j}]_{\mathcal{G}}|.
\end{aligned}
\end{equation}
	
	By Lemma 2.2, we have $\sum_{i=1}^{t}\sum_{\substack{j=1\\ j\neq i}}^{t}|[X_{i},\overline{X}_{j}]_{\mathcal{G}}|\geq\frac{t(t-1)}{2}2\lambda(G)= t(t-1)\delta(G).$
Therefore, by (3), we have
\begin{equation}
\begin{aligned}
	|S|\geq t \delta(G)+\sum_{i=1}^{t}|X_{i}|(\delta(G)+1)(n-t)+t(t-1)\delta(G).
\end{aligned}
\end{equation}

	If $|V(D_1)|=3$ or $|V(D_2)|=3$, then we have $|S|\geq\xi_3(\mathcal{G})=3n\delta(G)+3n-9$. So, in the following, we only consider the case where $|V(D_1)| \geq 4$ and $|V(D_2)| \geq 4$.

	\noindent{\bf Case 1.} $t=1$.

	When \( t = 1 \), it follows that \( V(D_1) \subseteq V(G^{y_{s+1}}) \), which implies \( V(D_1) = X_1 \).
Since $D_1$ is a connected graph and $|V(D_1)| \geq 4$, we have $G^{y_{s+1}}[X_1]$ is connected and $|X_1| \geq 4$. By (4), we have
	\[
	\begin{aligned}
		|S|&\geq t\delta(G)+|X_1|(\delta(G)+1)(n-1)\\
		&\geq\delta(G)+4(\delta(G)+1)(n-1)\\
		&=3n\delta(G)+3n-9+(n-3)\delta(G)+n+5\\
		&>3n\delta(G)+3n-9.
	\end{aligned}
	\]

	\noindent{\bf Case 2.} $t=2$.

	When $t = 2$, we have $V(D_1) = X_1 \cup X_2$. Since $|V(D_1)| \geq 4$, it follows that $|X_1| + |X_2| \geq 4$. By (4), we have
	\[
	\begin{aligned}
	|S|&\geq t\delta(G) + (|X_1| + |X_2|)(\delta(G) + 1)(n - t) + t(t-1)\delta(G)\\
    &\geq 2\delta(G) +4(\delta(G) + 1)(n - 2) + 2\delta(G)\\
	&=3n\delta(G)+3n-9+(n-4)\delta(G)+n+1\\
	&>3n\delta(G)+3n-9.
	\end{aligned}
	\]

\noindent{\bf Case 3.} $3\leq t\leq n - 1$.

\noindent{\bf Subcase 3.1.} For each $i\in\{1,\ldots,t\}$, $|X_i|\geq 2$.

	Since $|X_i| \geq 2$, it follows that $\sum_{i=1}^{t} |X_{i}| \geq 2t$. Therefore, by (4), we obtain	
\begin{align}
|S| &\geq t\delta(G) + 2t(\delta(G) + 1)(n - t) + t(t - 1)\delta(G) \nonumber \\
    &= t\delta(G) + 2tn\delta(G) - 2t^2\delta(G) + 2tn - 2t^2 + t^2\delta(G) - t\delta(G) \nonumber \\
    &= (-t^2 + 2tn)\delta(G) - 2t^2 + 2tn. \tag{5}
\end{align}

	Let \( f_1(t) = -t^2 + 2tn \) and \( f_2(t) = -2t^2 + 2tn \). Since \( 3 \leq t \leq n - 1 \) and \( n \geq 4 \), we have \( f_1(t) \geq \min\{f_1(3), f_1(n-1)\} = \min\{6n - 9,  n^2-1\}=6n - 9 \) and \( f_2(t) \geq \min\{f_2(3), f_2(n-1)\} = \min\{6n - 18, 2n - 2\} = 2n - 2 \). If $n = 4$, then $3n\delta(G) + 3n - 9 = 12\delta(G) + 3$ and $f_1(3) = 15$ and $f_2(3) = 6$. Thus, by (5), it follows that $|S| \geq 15\delta(G) + 6 > 12\delta(G) + 3.$
If \( n \geq 5 \), by (5), then $|S| \geq (6n-9)\delta(G) + 2n-2>4n\delta(G) + 2n -2=3n\delta(G) + 3n -9+n(\delta(G) -1)+7
>3n\delta(G)+3n-9.$

	\noindent{\bf Subcase 3.2.} There exist at least three integers in $\{1,\ldots,t\}$, denoted as 1, 2, and 3, such that $|X_1|=|X_2|=|X_3|=1$.

	Let $\sum_{i=4}^{t} |X_i| = M$. It follows that $\sum_{i=1}^{t}| X_i| = M + 3\geq t$. Considering that $|V(D_1)| \geq 4$, we have $M \geq 1$. Since $|X_1| = 1$, for any $i\in\{2, ..., t\}$, we have $|[X_{i}, \overline{X}_{1} ]_{\mathcal{G}}| \geq |X_{i}|(\delta(G)+1)-|X_{i}|=|X_{i}|\delta(G)$. Thus, $\sum_{i=2}^{t} | [ X_{i}, \overline{X}_{1} ]_{\mathcal{G}}| \geq \sum_{i=2}^{t} |X_{i}|\delta(G)= (M+2)\delta(G)$.
Analogously, since $|X_2| = |X_3| = 1$, we obtain $\sum_{\substack{i=1\\ i\neq 2}}^{t}| [ X_{i}, \overline{X}_{2}]_{\mathcal{G}}| \geq (M+2)\delta(G)$ and $\sum_{\substack{i=1\\ i\neq 3}}^{t}| [ X_{i}, \overline{X}_{3} ]_{\mathcal{G}}| \geq (M+2)\delta(G)$.
By Lemma 2.2, we obtain
	\[
	\sum_{i=4}^{t}\sum_{\substack{j=4\\ j\neq i}}^{t}|[X_{i},\overline{X}_{j}]_{\mathcal{G}}|\geq\frac{(t-3)(t-4)}{2}2\lambda(G)= (t-3)(t-4)\delta(G).
	\]
	Therefore,
	\begin{align*}
	\sum_{i=1}^{t} \sum_{\substack{j=1 \\ j \neq i}}^{t} |[ X_{i}, \overline{X}_{j}]_{\mathcal{G}}|
	&=\sum_{i=2}^{t} | [ X_{i}, \overline{X}_{1} ]_{\mathcal{G}}| + \sum_{\substack{i=1 \\ i \neq 2}}^{t}| [ X_{i}, \overline{X}_{2} ]_{\mathcal{G}}|+ \sum_{\substack{i=1 \\ i \neq 3}}^{t} | [ X_{i}, \overline{X}_{3} ]_{\mathcal{G}}|\\
	&+\sum_{i=1}^{t}\sum_{\substack{j=4\\ j\neq i}}^{t}|[X_{i},\overline{X}_{j}]_{\mathcal{G}}|\\
	&\geq \sum_{i=2}^{t} | [ X_{i}, \overline{X}_{1} ]_{\mathcal{G}}| + \sum_{\substack{i=1 \\ i \neq 2}}^{t}| [ X_{i}, \overline{X}_{2} ]_{\mathcal{G}}|+ \sum_{\substack{i=1 \\ i \neq 3}}^{t} | [ X_{i}, \overline{X}_{3} ]_{\mathcal{G}}|\\
	&+\sum_{i=4}^{t}\sum_{\substack{j=4\\ j\neq i}}^{t}|[X_{i},\overline{X}_{j}]_{\mathcal{G}}|\\
	&\geq 3(M+2)\delta(G) + (t-3)(t-4)\delta(G).
\end{align*}
	Since $t \leq n-1$, we have $Mn - Mt \geq M$. Thus, by (3), we have
	\begin{align}
	|S|&\geq t \delta(G) + (M+3)(\delta(G)+1)(n-t) + 3(M+2)\delta(G)+(t-3)(t-4)\delta(G)	\nonumber \\
	&=(t^2-9t+18+Mn-Mt+3M+3n)\delta(G)+Mn-Mt+3n-3t
	\nonumber \\
	&\geq(t^2-9t+18+4M+3n)\delta(G)+M+3n-3t.  \tag{6}
\end{align}

	Let $f_3(t) = t^2 - 9t + 18$. Since $3 \leq t \leq n - 1$ and $n \geq 4$, we have $f_3(t) \geq \min\{f_3(3), f_3(n-1)\} = \min\{0, n^2 - 11n + 28\}$. Let $f_4(t) = n^2 - 11n + 28$. Since $f_4(t) \geq f_4(5) = -2$, it follows that $f_3(t) \geq \min\{0, -2\} = -2$. By $M+3 \geq t$ and $\delta(G) \geq 1$, we obtain $3M \geq 3t - 9$ and $2M\delta(G) \geq 2M$. Since $M \geq 1$, we have $4M - 2 \geq 2M$.
Therefore, by (6),
$|S|\geq(-2+4M+3n)\delta(G)+M+3n-3t\geq (2M+3n)\delta(G)+M+3n-3t\geq 2M+3n\delta(G)+M+3n-3t=3n\delta(G)+3n-9+3M-3t+9\geq3n\delta(G)+3n-9+3t-9-3t+9=3n\delta(G)+3n-9. $

	\noindent{\bf Subcase 3.3.} There exist only two integers in $\{1,\ldots,t\}$, denoted as 1 and 2, such that $|X_1|=|X_2|=1$.

	For each \( i \in \{3, \ldots, t\} \), we have \( |X_i| \geq 2 \). Thus, it follows that $\sum_{i=1}^{t} |X_i| \geq 2t - 2.$
If $\sum_{i=1}^{t} |X_i| \geq 2t $, then we have $|S| \geq 3n\delta(G) + 3n - 9$ by a similar argument as Subcase 3.1.
So we consider the following two subcases.

	\noindent{\bf Subcase 3.3.1.} $\sum_{i=1}^{t}|X_{i}|= 2t-1$.

Since $|X_1|=|X_2|=1$ and $|X_i| \geq 2$ for  $i \in \{3, \ldots, t\}$, there must exist an integer in $\{3, \ldots, t\}$, denoted as $3$, such that $|X_3|=3$, and for each $i \in \{4, \ldots, t\}$, $|X_i|=2$.

Since $|X_1|=|X_2|=1$, we have $\sum_{i=2}^{t} |[X_{i}, \overline{X}_{1}]_{\mathcal{G}}| \geq \sum_{i=2}^{t} |X_{i}|\delta(G)= (2t-2)\delta(G)$ and $\sum_{\substack{i=1\\ i\neq 2}}^{t}| [ X_{i}, \overline{X}_{2}]_{\mathcal{G}}| \geq \sum_{\substack{i=1\\ i\neq 2}}^{t}|X_{i}|\delta(G)= (2t-2)\delta(G)$.
	By Lemma 2.2, we obtain
	\[
	\sum_{i=3}^{t}\sum_{\substack{j=3\\ j\neq i}}^{t}|[X_{i},\overline{X}_{j}]_{\mathcal{G}}|\geq\frac{(t-2)(t-3)}{2}2\lambda(G)= (t-2)(t-3)\delta(G).
	\]
	Therefore,
	\begin{align*}
	\sum_{i=1}^{t} \sum_{\substack{j=1 \\ j \neq i}}^{t} | [ X_i, \overline{X}_j ]_{\mathcal{G}}|
	&=\sum_{i=2}^{t} |[X_{i}, \overline{X}_{1}]_{\mathcal{G}}|+ \sum_{\substack{i=1\\ i\neq 2}}^{t}
	|[ X_{i}, \overline{X}_{2}]_{\mathcal{G}}|+ \sum_{i=1}^{t} \sum_{\substack{j=3 \\ j \neq i}}^{t} | [ X_i, \overline{X}_j ]_{\mathcal{G}} | \\
	& \geq \sum_{i=2}^{t} |[X_{i}, \overline{X}_{1}]_{\mathcal{G}}|+ \sum_{\substack{i=1\\ i\neq 2}}^{t}
	|[ X_{i}, \overline{X}_{2}]_{\mathcal{G}}|+\sum_{i=3}^{t} \sum_{\substack{j=3 \\ j \neq i}}^{t} | [ X_i, \overline{X}_j ]_{\mathcal{G}} | \\
	& \geq 4(t-1) \delta(G) + (t-2)(t-3) \delta(G).
\end{align*}
	
By (3), we have
	\begin{align}
	|S|&\geq t \delta(G) + (2t-1)(\delta(G)+1)(n-t) + 4(t-1)\delta(G)+(t-2)(t-3)\delta(G)\nonumber \\
	&\geq(-t^2+t+2tn-n+2)\delta(G)+2tn-2t^2+t-n.\tag{7}
	\end{align}

	Let $f_5(t)= -t^2+t+2tn-n+2$ and $f_6(t)=2tn-2t^2+t-n$. Since $3\leq t\leq n - 1$ and $n\geq 4$, we have $f_5(t)\geq\min\{f_5(3),f_5(n - 1)\}=\min\{5n - 4,n^2\}=5n-4$ and $f_6(t)\geq\min\{f_6(3),f_6(n - 1)\}=\min\{5n - 15,2n-3\}=2n-3$. Therefore, by (7), $|S|\geq (5n-4)\delta(G)+2n-3\geq 4n\delta(G)+2n-3=3n\delta(G)+3n-9+n(\delta(G)-1)+12> 3n\delta(G)+3n-9. $

	\noindent{\bf Subcase 3.3.2.} $\sum_{i=1}^{t} |X_{i}| = 2t - 2$.

Since $|X_1|=|X_2|=1$ and $|X_i| \geq 2$ for $i \in \{3, \ldots, t\}$, it follows that $|X_i| = 2$ for $i \in \{3, \ldots, t\}$.
By $|X_1|=|X_2|=1$, we have \(\sum_{\substack{i = 1\\i\neq j}}^{t}|X_{i},\overline{X}_{j}]_{\mathcal{G}}| \geq \sum_{\substack{i = 1\\i\neq j}}^{t}|X_{i}|\delta(G)=(2t-3)\delta(G)\) for $j=1$ and $j=2$. By $|X_3|=...=|X_t|=2$, we have \(\sum_{\substack{i = 1\\i\neq j}}^{t}|X_{i},\overline{X}_{j}]_{\mathcal{G}}| \geq \sum_{\substack{i = 1\\i\neq j}}^{t}|X_{i}|(\delta(G)-1)=(2t-4)(\delta(G)-1)\) for each $j \in \{3, \ldots, t\}$. Therefore,
	\[
\begin{aligned}
	\sum_{i = 1}^{t}\sum_{\substack{j = 1\\j\neq i}}^{t}|[X_{i},\overline{X}_{j}]_{\mathcal{G}}|& = \sum_{i = 2}^{t}|[X_{i},\overline{X}_{1}]_{\mathcal{G}}|+\sum_{\substack{i = 1\\i\neq 2}}^{t}|[X_{i},\overline{X}_{2}]_{\mathcal{G}}| + \sum_{j=3}^{t}\sum_{\substack{i = 1\\i\neq j}}^{t}|[X_{i},\overline{X}_{j}]_{\mathcal{G}}| \\
	&\geq (2t - 3)\delta(G)+(2t - 3)\delta(G) + (2t - 4)(\delta(G) - 1)(t - 2)\\
	&= 2(2t - 3)\delta(G) + 2{(t - 2)^2}(\delta(G) - 1).
\end{aligned}
\]
	By (3), we have
	\begin{align}
	|S|&\geq t \delta(G) + (2t-2)(\delta(G)+1)(n-t) + 2(2t - 3)\delta(G) + 2(\delta(G) - 1){(t - 2)^2}.\tag{8}
	\end{align}

If $\delta(G)=1$, then $3n\delta(G)+3n-9=6n-9$. By (8), we have $|S|\geq-4t^2+9t+4tn-4n-6$. Let $f_7(t)=
-4t^2+9t+4tn-4n-6$. Since $3\leq t\leq n - 1$ and $n\geq 4$, $f_7(t)\geq\min\{f_7(3),f_7(n - 1)\}=\min\{8n - 15,9n-19\}=8n-15>6n-9$. Therefore, $|S|>6n-9$.

If $\delta(G)\geq2$, by (8), then $|S|\geq t \delta(G) + (2t-2)(\delta(G)+1)(n-t) + 2(2t - 3)\delta(G) + 2(t - 2)^2=(-2t^2+2tn+7t-2n-6)\delta(G)+2tn-6t-2n+8$. Let $f_8(t)=-2t^2+2tn+7t-2n-6$ and $f_9(t)=2tn-6t-2n+8$. Since $3\leq t\leq n - 1$ and $n\geq 4$, we have $f_8(t)\geq\min\{f_8(3), f_8(n - 1)\}=\min\{4n-3, 7n-15\}=4n-3$ and $f_9(t)\geq f_9(3)=4n-10$. Therefore, $|S|\geq(4n-3)\delta(G)+4n-10>3n\delta(G)+3n-9+n-1>3n\delta(G)+3n-9$.

\noindent{\bf Subcase 3.4.} There exists only one integer in $\{1,\ldots,t\}$, denoted as 1, such that $|X_1|=1$.

If there exists some $i \in \{2, \ldots, t\}$ such that $|X_i|\geq3$, then $\sum_{i=1}^{t} |X_{i}| \geq 2t$. By a similar argument as Subcase 3.1, we obtain $|S|> 3n\delta(G)+3n-9$. Thus, assume that $|X_i|=2$ for each $i\in\{2, \ldots, t\}$. By $|X_1|=1$, we have $\sum_{i=2}^{t} | [ X_{i},  \overline{X}_{1} ]_{\mathcal{G}} |\geq \sum_{i=2}^{t}  |X_{i}|\delta(G)=(2t-2)\delta(G)$. By $|X_2|=...=|X_t|=2$, we have \(\sum_{\substack{i = 1\\i\neq j}}^{t}|X_{i},\overline{X}_{j}]_{\mathcal{G}}| \geq \sum_{\substack{i = 1\\i\neq j}}^{t}|X_{i}|(\delta(G)-1)=(2t-3)(\delta(G)-1)\) for each $j \in \{2, \ldots, t\}$. Therefore,
\begin{align*}
	\sum_{i=1}^{t} \sum_{\substack{j=1 \\ j \neq i}}^{t} | [ X_{i},  \overline{X}_{j} ]_{\mathcal{G}} |
	&= \sum_{i=2}^{t} | [ X_{i},  \overline{X}_{1} ]_{\mathcal{G}} | + \sum_{j=2}^{t} \sum_{\substack{i=1 \\ i \neq j}}^{t} | [ X_{i},  \overline{X}_{j} ]_{\mathcal{G}} | \\
	&\geq (2t-2)\delta(G) + (2t-3)(\delta(G)-1)(t-1).
\end{align*}
By (3), we have
	\begin{align}
	|S|&\geq t \delta(G) + (2t-1)(\delta(G)+1)(n-t) + (2t - 2)\delta(G) +  (2t-3)(\delta(G)-1)(t-1).\tag{9}
	\end{align}
	
If $\delta(G)=1$, then $3n\delta(G)+3n-9=6n-9$. By (9), we have $|S|\geq-4t^2+4tn+5t-2n-2$. Let $f_{10}(t)=
-4t^2+4tn+5t-2n-2$. Since $3\leq t\leq n - 1$ and $n\geq 4$, $f_{10}(t)\geq\min\{f_{10}(3),f_{10}(n - 1)\}=\min\{10n - 23,7n-11\}=7n-11>6n-9$. Therefore, $|S|>6n-9$.

	If $\delta(G)\geq2$, by (9), then $|S|\geq t \delta(G) + (2t-1)(\delta(G)+1)(n-t) + (2t - 2)\delta(G) +(2t-3)(t-1)=(-2t^2+2tn+4t-n-2)\delta(G) + 2tn-4t-n+3$. Let $f_{11}(t)=-2t^2+2tn+4t-n-2$ and $f_{12}(t)=2tn-4t-n+3$. Since $3\leq t\leq n - 1$ and $n\geq 4$, we have $f_{11}(t)\geq\min\{f_{11}(3),f_{11}(n - 1)\}=5n-8$ and $f_{12}(t)\geq f_{12}(3)=5n-9$. Therefore, $|S|\geq(5n-8)\delta(G)+5n-9\geq3n\delta(G)+3n-9+2n>3n\delta(G)+3n-9$.
$\square$

Since $\xi_3(G \boxtimes K_n)=3n\delta(G) + 3n - 9$, by Theorem 3.2, we obtain the following corollary.

\begin{corollary}
Let \( G \) be a maximally edge-connected graph of order $m \geq 3$. If $\min\{ n^{2}\delta(G), (n-1)(m +2e(G))\}
\geq 3n\delta(G) + 3n - 9$, then \( G \boxtimes K_n \) is maximally 3-restricted edge-connected, where \( n \geq 4 \).
\end{corollary}

\section{Concluding Remarks}
In this paper, we determine the 3-restricted edge-connectivity for the strong product graphs of a maximally edge-connected graph with the cycle and the complete graph. Furthermore, we establish sufficient conditions under which these strong product graphs are maximally 3-restricted edge-connected. Future research may focus on the 3-restricted edge-connectivity for the strong product graphs of a connected general graph with the cycle and the complete graph.



\end{sloppypar}

\begin{thebibliography}{5}

\bibitem{Bai} M. Bai, Y. Tian, J. Yin, The super restricted edge-connectedness of direct product graphs, Parallel Processing Letters 33(3) (2023) 2350008.

\bibitem{Bondy} J. A. Bondy, U. S. R. Murty, Graph Theory, Springer, New York, 2008.

\bibitem{Bresar1} B. Bre{\v{s}}ar, S. {\v{S}}pacapan, Edge-connectivity of Strong products of graphs, Discussiones Mathematicae Graph Theory 27(2) (2007) 333-343.

\bibitem{Bresar2} B. Bre{\v{s}}ar, S. {\v{S}}pacapan, On the connectivity of the direct product of graphs, Australasian Journal of Combinatorics 41 (2008) 45-56.

\bibitem{Cao} X. L. Cao, {\v{S}}. Brglez, S. {\v{S}}pacapan, E. Vumar, On edge connectivity of direct products of graphs, Information Processing Letters 111(18) (2011) 899-902.

\bibitem{Esfahanian} A. H. Esfahanian, S. L. Hakimi, On computing a conditional edge-connectivity of a graph, Information Processing Letters 27(4) (1988) 195-199.

\bibitem{Guo} S. Guo, X. Hu, W. Yang, S. Zhao, The super edge-connectivity of direct product of a graph and a cycle, The Journal of Supercomputing 80(16) (2024) 23367-23383.

\bibitem{Harary} F. Harary, Conditional connectivity, Networks 13(3) (1983) 347-357.

\bibitem{Klavzar} S. Klav{\v{z}}ar, S. {\v{S}}pacapan, On the edge-connectivity of Cartesian product graphs, Asian-European Journal of Mathematics 1(1) (2008) 93-98.

\bibitem{Ma} T. Ma, J. Wang, M. Zhang, The restricted edge-connectivity of Kronecker product graphs, Parallel Processing Letters 29(3) (2019) 1950012.

\bibitem{Ou1} J. Ou, On optimizing edge connectivity of product graphs, Discrete Mathematics 311(6) (2011) 478-492.

\bibitem{Ou2} J. Ou, W. Zhao, On restricted edge connectivity of strong product graphs, Ars Combinatoria 123 (2015) 55-64.

\bibitem{Spacapan} S. {\v{S}}pacapan, A characterization of the edge connectivity of direct products of graphs, Discrete Mathematics 313(12) (2013) 1385-1393.

\bibitem{Wang1} Y. Wang, Q. Li, Upper bound of the third edge-connectivity of graphs, Science in China Series A: Mathematics 48(3) (2005) 360-371.

\bibitem{Wang2} J. Wang, J. Ou, T. Zhu, On restricted edge connectivity of regular Cartesian product graphs, Australasian Journal of Combinatorics 48 (2010) 111-116.

\bibitem{Wang3} Z. Wang, Y. Mao, C. Ye, H. Zhao, Super edge-connectivity of Strong product graphs, Journal of Interconnection Networks 17(2) (2017) 1750007.

\bibitem{Yang} C. Yang, J. M. Xu, Connectivity and edge-connectivity of Strong product graphs, University of Science and Technology of China 38(5) (2008) 449-455.

\bibitem{Ye} H. Ye, Y. Tian, The Restricted Edge-Connectivity of Strong Product Graphs, Axioms 13(4) (2024) 231.

\end{thebibliography}
\end{document}